\documentclass[a4paper,12pt]{article}

\usepackage{amssymb}
\usepackage{amsfonts,dsfont}
\usepackage{amsmath}
\usepackage{amsthm}
\usepackage{cite}
 \usepackage{enumitem}
\numberwithin{equation}{section}
\usepackage{amsmath}

\usepackage{xcolor}

\newcommand{\R}{\mathbb{R}}
\newcommand{\Z}{\mathbb{Z}}
\newcommand{\N}{\mathbb{N}}
\newcommand{\cuad}{{\sqcap\kern-.68em\sqcup}}

\newcommand{\norm}[1]{\|#1\|}

\newtheorem{definition}{Definition}[section]
\newtheorem{theorem}[definition]{Theorem}
\newtheorem{proposition}[definition]{Proposition}

\newtheorem{lemma}[definition]{Lemma}
\newtheorem{corollary}[definition]{Corollary}
\newtheorem{remark}{Remark}[section]

\renewcommand{\phi}{\varphi}

\newcommand{\cE}{{\mathcal E}}

\newcommand{\cL}{{\mathcal L}}

\newcommand{\cP}{{\mathcal P}}

\newcommand{\bH}{{\mathbb H}}

 \allowdisplaybreaks

\begin{document}

\begin{center}
{\bf \large   Dirichlet Eigenvalues  for the Logarithmic-Hardy operator }
\bigskip
\medskip

{\small
Tangrui Liao\footnote{liaotangrui@163.com} \qquad  Ying Wang\footnote{yingwang00@126.com}\qquad
Minghua Yang\footnote{minghuayang@jxufe.edu.cn}

\bigskip \medskip

School of Information Management and Mathematics, Jiangxi University \\
of  Finance and Economics,  Nanchang, Jiangxi 330032,  PR China\\[10pt]

}
\begin{abstract}\small
In this paper, we study the Dirichlet eigenvalue problem for the logarithmic Hardy operator, defined as the logarithmic Laplacian with  a critical  logarithmic potential, in a bounded Lipschitz domain containing the origin. We first establish a logarithmic Hardy inequality and use it to identify the associated energy space, showing that the embedding into \(L^2\) is compact when the perturbation parameter exceeds the endpoint value \(-1\), but fails to be compact at that endpoint. For parameters above \(-1\), we develop a complete variational spectral theory, prove that  the problem admits a discrete sequence of eigenvalues tending to infinity, characterized by successive minimizations over orthogonal complements, whose eigenfunctions form a complete orthonormal basis of \(L^2\).

We further establish a uniform lower bound for the first eigenvalue, a scaling property under domain dilation that determines a critical radius for the sign of the first eigenvalue and leaves the eigenvalue gaps invariant.  
\end{abstract}
\end{center}
{\footnotesize


}
\bigskip

\setcounter{equation}{0}
\section{Introduction}
Our purpose of this paper is to study the eigenvalue problem
\begin{equation}\label{eq: 1.1}
 \left\{
\begin{array}{lll}
\displaystyle  \cL_\mu u= \lambda u\quad
 & {\rm in}\  \   \Omega,
\\[2mm]
 \phantom{ -\ \,   }
 \displaystyle u=0 \quad
 & {\rm in}\  \, \R^N\setminus  \Omega,
 \end{array}
 \right.
 \end{equation}
 where  $0\in \Omega\subset \R^N$ is a  bounded Lipschitz domain with $N\geq 1$, $\mu\geq -1$
 and $\cL_\mu$ is the logarithmic Hardy operator defined by
 $$\cL_\mu u=(-\Delta)^{\ln} u+2\mu\big (\log\frac1{|x|}\big) u,$$
the operator $(-\Delta)^{\ln}$ is the logarithmic Laplacian whose Fourier symbol is $2\ln|\xi|$ and
which arises formally as the derivative at $s=0$ of the fractional Laplacian $(-\Delta)^s$, Chen and Weth \cite{ChenWeth2019} gives its expression
 \begin{align*}
  (-\Delta)^{\ln} u(x)&= \left.\frac{d}{ds}\right|_{s=0}[(-\Delta)^su](x)
  \\[2mm]&=c_N\,  {\rm p.v.} \int_{B_1(x)}\frac{u(x)-u(y)}{|x-y|^{N}}\,dy	
  - c_N\int_{\mathbb{R}^N\setminus B_1(0)}\frac{u(x+y)}{|y|^N}\,dy+\rho_N\,u(x),
\end{align*}
for $u\in C_c^\beta$ and some $\beta>0$,
where $B_1(x)$ denotes the open unit ball in $\mathbb{R}^N$ centred at  $x$,
$$
	c_N=\frac{\Gamma(\frac N2)}{\pi^{\frac N2}},
	\qquad \
	\rho_N=2\ln 2+\psi\big(\frac N2\big)-\gamma,
$$
$\Gamma$ denotes the Gamma function,	 $\gamma=-\Gamma'(1)$ is the Euler--Mascheroni constant,
$\psi=\frac{\Gamma'}{\Gamma}$ is the Digamma function.

Following the seminal work of Chen and Weth \cite{ChenWeth2019},  this structural distinction underscores the intrinsic nature of $(-\Delta)^{\ln}$ as a genuinely nonlocal operator of zero order, rather than a limiting case of fractional powers. Consequently, understanding its fundamental properties including its spectral characterization, integral representation, and conformal behavior has become increasingly important across multiple mathematical contexts.   Specifically, $(-\Delta)^{\ln}$ arises naturally in:
(i) the asymptotic analysis of the Dirichlet eigenvalue problem for $(-\Delta)^s$ as $s \to 0^+$, see \cite{FeuJa};
(ii) $s$-dependent nonlinear Dirichlet problems in the small-$s$ regime, motivated by optimization models where the optimal order is small, in image processing and population dynamics,
see \cite{sprekels.valdinoci,pellacci.verzini,hs.saldana};
and (iii) the geometric study of the $0$-fractional perimeter, see \cite{DNP} and references therein.
Chen, Hauer and Weth in \cite{CDW} established that the logarithmic Laplacian $(-\Delta)^{\ln}$ admits
 a unique extension property distinct from the Caffarelli--Silvestre extension for the fractional Laplacian $(-\Delta)^s$ in \cite{Caffarelli-Silvestre}. \smallskip

When $\mu=0$, quantitative eigenvalue bounds and further spectral properties for the
unperturbed logarithmic Laplacian were subsequently studied in
\cite{ChenVeron2023,LaptevWeth2021,Chen2026}. Feulefack, Jarohs and Weth
\cite{FeuJa} built relate small order asymptotics of fractional Dirichlet eigenvalues and subsequential
limits of normalized eigenfunctions to the spectral data for the
logarithmic Laplacian. Critical singular potentials also play a central
role in fractional spectral inequalities in
\cite{FrankLiebSeiringer2008}.
Recently, Arora, Mukherjee and Vaishnavi in \cite{AroraMukherjeeVaishnavi2026}
develop variational spectral theory for the indefinite-weight problem
$(-\Delta)^{\ln} u=\lambda\omega(x)u$, using a logarithmic Hardy--Pitt
estimate, where $\omega$ is a subcritical weight.  In contrast to the weight multiplying the spectral parameter
in that problem (\ref{eq: 1.1}) contains the additive
singular perturbation $\mu(\log\frac1{|x|})u$.

The aim of this paper is to study the eigenvalue problem  (\ref{eq: 1.1}) with $\mu\geq-1$, to this end,
we first introduce the related spaces.  Let
 $$\bH^{\ln}_0 (\Omega)= \overline{C_c^\infty(\Omega)}^{\|\cdot\|_{\ln }},$$
 where
$$\|u\|_{\ln}=\Big(c_N\int\!\!\!\!\int_{\{x,y\in\R^N: \, |x-y|<1\}}\frac{\big(u(x)-u(y)\big)^2}{|x-y|^{N}}\,dxdy +\int_{\Omega} u(x)^2\,dx\Big)^{\frac12}  \, .$$
Note that $\bH^{\ln}_0 (\Omega)$ is a Hilbert space with the inner product
$$ \langle u,v\rangle_{\ln}=c_N
\int\!\!\!\!\int_{\{x,y\in\R^N: \, |x-y|<1\}}\frac{\big(u(x)-u(y)\big)\big(v(x)-v(y)\big) }{|x-y|^{N}}\, dxdy
+\int_{\Omega} u(x)v(x)\,dx.   $$

We have the following Hardy type inequality.
\begin{proposition}
\label{lem:domain-hardy}
Let  $  \Omega\subset \R^N$ be a  bounded Lipschitz domain with $N\geq 1$ containing the origin.
For every $u\in C_c^\infty(\Omega)$, then we have
\begin{equation}\label{eq:smooth-hardy-control}
2\int_\Omega\big(\log\frac1{|x|}\big)\,u(x)^2\,dx+ d_{N,\Omega}   \norm{u}_{2}^2\leq \|u\|_{\ln}^2,
\end{equation}
where
$$d_{N,\Omega}= 2\psi\!\big(\frac N4 \big)-2\psi\!\big(\frac N2 \big) +1-\gamma-C_N |\Omega|^{\frac12}$$
and
$C_N=c_N(\int_{\R^N\setminus B_1(0)} |y|^{-2N}dy)^{\frac12}\, .$
\end{proposition}

Denote
\[
\langle u,v\rangle_{\ln,\mu} = \langle u,v\rangle_{\ln} +2\mu \int_{\Omega}  \big(\log\frac1{|x|} \big)_+u(x) v(x)\,dx+ (d_{N,\Omega})_+ \int_{\Omega}  u(x) v(x)\,dx
\]
for $u,v\in C^1_c(\Omega)$ and
its  induced norm
$$\|u\|_{\ln,\mu}=\sqrt{ \langle u,u\rangle_{\ln,\mu} } \,  , $$
where  $a_\pm=\max\{0,\pm a\}$.
Hence,
\[
\bH_0^{\ln,\mu}(\Omega)
:=
\overline{C_c^\infty(\Omega)}^{\|\cdot\|_{\ln,\mu}}
\]
 is a Hilbert space with above inner product.

\begin{theorem}
	\label{prop:energy-space}
	  Let  $  \Omega\subset \R^N$ be a  bounded Lipschitz domain with $N\geq 1$ containing the origin.
Then
		\begin{equation}\label{eq:space-identification}
			\bH_0^{\ln,\mu}(\Omega)\cong \bH_0^{\ln}(\Omega)\qquad {\rm for}\ \ \mu>-1,
			\end{equation}
	and the embedding
		\begin{equation}\label{eq:perturbed-compact}
			\mathbb H_0^{\ln,\mu}(\Omega)\hookrightarrow L^2(\Omega)
		\end{equation}
		is continuous for $\mu\geq -1$
		and compact for $\mu>-1$ and non-compact for $\mu=-1$.
\end{theorem}
The embedding (\ref{eq:perturbed-compact}) is compact  for $\mu>-1$ and it fails  for $\mu=-1$,
which is different from the fractional Hardy inequality where the compactness of the related embedding is valid for the extremal case. \smallskip

Now we are ready to show the eigenvalue and corresponding eigenfunction for problem (\ref{eq: 1.1}).
A function $u\in \bH_0^{\ln,\mu}(\Omega)$ is called the eigenfunction of (\ref{eq: 1.1})  corresponding to the eigenvalue  $\lambda$ if
\begin{equation}\label{weak-s}
\mathcal E_{\ln,\mu}(u,\phi)= \lambda\int_{\Omega}u \,\phi \, dx,\qquad \forall  \, \phi\in \bH_0^{\ln,\mu}(\Omega),
\end{equation}
where the Dirichlet energy functional of $\cL_\mu$ in $\Omega$ has the expression
\begin{align}\label{Euv1}
  \cE_{\ln,\mu}(u,\phi)
    =\int_{\Omega}u(x)(-\Delta)^{\ln}\phi(x)\,dx
    +2\mu \int_{\Omega} \Big(\log\frac1{|x|}\Big) u(x)\phi(x) \,dx,
\end{align}
which, we will show later, is well-defined for  $u\in \bH_0^{\ln,\mu}(\Omega)$.

We state the characterization of Dirichlet eigenvalues and corresponding eigenfunctions for the operator  $\cL_\mu$
as follows.

 \begin{theorem}\label{teo 1-m}
Assume that $N\geq 1$,  $\Omega\subset \R^N$ is a bounded  Lipschitz continuous  domain  containing the origin
and $\mu>-1$.

\begin{itemize}
\item[(i)] Then problem (\ref{eq: 1.1})   admits  an eigenvalue $\lambda^{\ln,\mu}_{1}(\Omega)>0$   characterized  by
\begin{equation}\label{Lambda-1-s}
\lambda^{\ln,\mu}_{1}(\Omega)=\inf_{\substack{0\neq u\in \bH^{\ln,\mu}_{0}(\Omega)}}\frac{\mathcal E_{\ln,\mu}(u,u)}{\|u\|^2_{2}}= \inf_{u\in \cP^{\ln,\mu}_{1}(\Omega)}\mathcal E_{\ln,\mu}(u,u),
\end{equation}
with $$\cP^{\ln,\mu}_{1}(\Omega):=\{u\in \bH^{\ln,\mu}_{0}(\Omega): \|u\|_{2}=1\},$$
 and there exists a  function $\phi^{\ln,\mu}_{1}\in \cP^{\ln,\mu}_{1}(\Omega)$ achieving the minimum of $\mathcal E_{\ln,\mu}$, i.e.
 $
 \lambda^{\ln,\mu}_{1}(\Omega)= \mathcal E_{\ln,\mu}(\phi^{\ln,\mu}_{1},\phi^{\ln,\mu}_{1})$.
Furthermore, $\phi^{\ln,\mu}_{1}>0$ a.e. in $\Omega$.

\item[(ii)]
  If $\Omega\subset B_r(0)$ with $r>0$ small enough, then $\lambda^{\ln,\mu}_{1}(\Omega)\geq0$.

\item[(iii)] Then problem (\ref{eq: 1.1})   admits a sequence of eigenvalues $\{\lambda_{m,k}(\Omega)\}_{k\in \N}$
satisfying
\[
-\infty< \lambda^{\ln,\mu}_{1}(\Omega)< \lambda^{\ln,\mu}_{2}(\Omega)\le \cdots\le \lambda^{\ln,\mu}_{k}(\Omega)\leq \cdots
\]
with corresponding eigenfunctions $\phi^{\ln,\mu}_{k}$, $k\in \N$ and
$\lim_{k\to \infty}\lambda^{\ln,\mu}_{k}(\Omega) = +\infty$.

Furthermore,  for any $k=2,3,\cdots$, the  eigenvalue $\lambda^{\ln,\mu}_{k}(\Omega)$ can be characterized as
\begin{equation}\label{Lambda-k-s}
\lambda^{\ln,\mu}_{k}(\Omega)= \inf_{u\in \cP^{\ln,\mu}_{k}(\Omega)}\cE_{\ln,\mu}(u,u),
\end{equation}
where
\[
\cP^{\ln,\mu}_{k}(\Omega):= \Big\{ u\in \bH^{\ln,\mu}_{0}(\Omega): \int_{\Omega}u\,
\phi^{\ln,\mu}_{j}dx  =0 \ \text{ for } j=1,2,\cdots k-1 \ \text{ and } \ \|u\|_{2}=1\Big\}.
\]
\item[(iv)] The sequence of eigenfunctions  $\{\phi^{\ln,\mu}_{k}\}_{k\in\N}$ corresponding to eigenvalues $\{\lambda^{\ln,\mu}_{k}(\Omega)\}$ form a complete orthonormal basis of $L^2(\Omega)$ and  an orthogonal system of $\bH^{\ln,\mu}_{0}(\Omega)$.

\end{itemize}
 \end{theorem}

Theorem \ref{teo 1-m} establishes a complete variational spectral
theory for $\cL_\mu$ with $\mu>-1$, where the spectrum is discrete, the
eigenvalues are given by successive minimizations over
$L^2$-orthogonal complements, and the eigenfunctions form a complete
orthonormal basis of $L^2(\Omega)$. In particular, the first eigenvalue
may be negative due to the logarithmic Hardy potential, which motivates
the scaling analysis below.


\begin{theorem}\label{teo 1.4}
Assume that  $N\geq 1$, $\Omega\subset \R^N$ is a bounded  Lipschitz continuous  domain  containing the origin
and $\mu>-1$. Then

$(i)$ lower bound for the first eigenvalue:
\begin{equation}\label{lower-bound-1}
  \lambda_1^{\ln,\mu}(\Omega)
  \;\geq\;
  2\Big(\psi\!\Big(\frac N4\Big)+\log 2\Big)
  -2(1+\mu)\log R_0,
\end{equation}
where $R_0=\sup_{x\in \Omega}|x|$;

$(ii)$  for every $k\in\N$ and $l>0$,
\begin{equation}\label{scale-1}
  \lambda_k^{\ln,\mu}(\Omega_l)
  =
  \lambda_k^{\ln,\mu}(\Omega)
  -2(1+\mu)\log l,
\end{equation}
where $\Omega_l=l\,\Omega=\{lx:x\in\Omega\}.$


\end{theorem}

From Theorem~\ref{teo 1.4}, part $(i)$ establishes a uniform lower bound for the eigenvalues of $\mathcal{L}_\mu$, in sharp contrast to the behavior of the fractional Laplacian $(-\Delta)^s$ with $s \in (0,1)$. Indeed, for $(-\Delta)^s$, the eigenvalues on scaled domains obey the scaling law
\[
  \lambda_k((-\Delta)^s,\Omega_l) = l^{-2s}\lambda_k((-\Delta)^s,\Omega),
\]
and thus converge to $0^+$ as $l \to +\infty$, indicating that eigenvalues of large domains shrink toward zero from above. In contrast, part $(ii)$ reveals a fundamentally different asymptotic regime: for all eigenvalues of both $\mathcal{L}_\mu$ and its Hardy perturbation (with $\mu > -1$), scaling the domain leads to eigenvalues diverging to $-\infty$, quantified precisely by
\[
  \lim_{l \to +\infty} \frac{\lambda_k^{\ln,\mu}(\Omega_l)}{\log l} = -2(1+\mu).
\]

The paper is organized as follows. Section 2 contains the preliminary
results on the logarithmic Hardy inequality, the identification of the
energy spaces, and the compactness of the embedding
$\bH_0^{\ln,\mu}(\Omega)\hookrightarrow L^2(\Omega)$ for $\mu>-1$,
together with its failure at the endpoint $\mu=-1$. In Section 3 we
construct the Dirichlet eigenvalues and eigenfunctions of
\eqref{eq: 1.1} by variational methods and then establish their main
spectral properties. Section 4 is devoted to further properties of the
eigenvalues, including a uniform lower bound for
$\lambda_1^{\ln,\mu}(\Omega)$ and the scaling property
  for dilated domains.  

\setcounter{equation}{0}
\section{Preliminary}

This section is devoted to introduce the logarithmic Hardy inequality in Proposition \ref{lem:domain-hardy},
the identification of energy spaces  and the embedding in Theorem \ref{prop:energy-space}.
To be convenience for the analysis, we denote
 $a_\pm=\max\{0,\pm a\}$, $R_0=\sup\{|x|: x\in\Omega\}$,
$$C_c^\infty(\Omega)=\{u\in C^\infty(\R^N):\, {\rm supp}(u)\subset \Omega \}. $$
and
$$\langle u,v\rangle_2=\int_{\Omega} u(x)v(x)dx,\qquad \norm{u}_2=\Big(\int_\Omega u(x)^2dx\Big)^{\frac12}.$$

\subsection{Logarithmic Hardy inequality }

In this subsection, we prove the logarithmic Hardy inequality in Proposition \ref{lem:domain-hardy}.

\begin{proof}[{\bf Proof of Proposition \ref{lem:domain-hardy}}]
We first recall  the Beckner's logarithmic uncertainty
inequality from \cite{Beckner1995}:
\begin{equation}\label{eq:beckner}
	\int_{\R^N}(\log|x|)\,|f(x)|^2\,dx
	+\int_{\R^N}(\log|\eta|)\,
	|\widehat f_B(\eta)|^2\,d\eta
	\geq
	\Big[\psi\!\big(\frac N4\big)-\log\pi\Big]
	\norm{f}_2^2,
\end{equation}
  where the Fourier transform convention has the form
\[
\widehat f_B(\eta)=\int_{\R^N}e^{2\pi {\bf i} x\cdot\eta}f(x)\,dx\quad \text{for $f\in C_c^\infty(\R^N)$}
\]
with  ${\bf i}^2=-1$.

Now we take a uniform of   Fourier transformation used in \cite{ChenWeth2019}:
\[
  \widehat f(\xi)=(2\pi)^{-N/2}\int_{\R^N}e^{-{\bf i}  x\cdot\xi}f(x)\,dx
  =(2\pi)^{-N/2}\widehat f_B\!\big(-\frac{\xi}{2\pi}\big).
\]
Using $\xi=-2\pi\eta$ and Plancherel's identity, it holds that
\begin{align*}
  \int_{\R^N}(\log|\xi|)
|\widehat f(\xi)|^2\,d\xi
  &=\int_{\R^N}\log(2\pi|\eta|)
|\widehat f_B(\eta)|^2\,d\eta\\[2mm]
  &=\int_{\R^N}\log|\eta|\,
|\widehat f_B(\eta)|^2\,d\eta
  +\log(2\pi)\norm{f}_2^2.
\end{align*}
Substituting this identity into \eqref{eq:beckner} yields
\begin{equation}\label{eq:whole-space-hardy}
  \int_{\R^N}\big(\log\frac1{|x|}\big)\,|f(x)|^2\,dx
  \leq
  \int_{\R^N}(\log|\xi|)\,
  |\widehat f(\xi)|^2\,d\xi
  -\Big(\psi\!\big(\frac N4\big)+\log2\Big)
  \norm{f}_2^2.
\end{equation}

Note that
\begin{align*}
&\quad \ \Big|c_N\int\!\!\!\!\int_{\{x,y\in\R^N: \, |x-y|<1\}}\frac{u(x)u(y)}{|x-y|^N}\,dx\,dy\Big|
\\[2mm]& \leq c_N\int_{\Omega} |u(x)|\big( \int_{\{y\in \Omega: \, |x-y|\geq 1\}} \frac{ |u(y)|}{|x-y|^N}\, dy\big)dx
\\[2mm]& \leq c_N \int_{\Omega} |u(x)|dx \Big (\int_{\Omega}u(y)^2dy\Big)^{\frac12}
\Big(\int_{\{\R^N\setminus B_1(0)}\frac1{|y|^{2N}}dy\Big)^{\frac12}
\\[2mm]& \leq C_N |\Omega|^{\frac12} \norm{u}_2^2,
\end{align*}
where $C_N=c_N\big(\int_{\{\R^N\setminus B_1(0)}\frac1{|y|^{2N}}dy)^{\frac12}>0$.

By (\ref{Euv1}) with $\mu=0$, note that
$$
2\int_{\R^N}(\log|\xi|)\,
|\widehat{  u}(\xi)|^2\,d\xi
=\cE_{\ln,0}(u,u),
$$
then we have that
\begin{align*}
2\int_\Omega\big(\log\frac1{|x|}\big)\,u(x)^2\,dx
&\leq
\cE_{\ln,0} (u,u) -2\Big(\psi\!\big(\frac N4 \big)+\log2\Big)\norm{u}_2^2
\\[2mm]&=\|u\|_{\ln}^2
    +c_N\int\!\!\!\int_{|x-y|\geq 1}\frac{u(x)u(y)}{|x-y|^N}\,dx\,dy
    \\[2mm]&\qquad
    +\Big( \rho_N-1 -2\psi\!\big(\frac N4\big)-2\log2\Big) \int_{\Omega} u(x)^2dx
    \\[2mm]&  \leq \|u\|_{\ln}^2 +\Big( \rho_N-1+C_N|\Omega|^{\frac12} -2\psi\!\big(\frac N4 \big)-2\log2\Big) \int_{\Omega} u(x)^2dx,
\end{align*}
 which implies  (\ref{eq:smooth-hardy-control}).
\end{proof}

\subsection{Embedding property for Hilbert space}

In this section, we show the proof about the identification of energy spaces  and the embedding
in Theorem \ref{prop:energy-space}.

\begin{proof}[{\bf Proof of Theorem \ref{prop:energy-space}}]
When $\mu\geq0$, the logarithmic Hardy inequality \eqref{eq:smooth-hardy-control} shows that
\begin{align*}
  \|u\|_{\ln,\mu}\geq \|u\|_{\ln}
\end{align*}
and by Proposition \ref{lem:domain-hardy}, it yields that
\begin{align*}
  \|u\|_{\ln,\mu}^2
  &= \|u\|_{\ln}^2+2\mu\int_{\Omega}\Big((\log\frac1{|x|})_+ u(x)^2\Big) + (d_{N,\Omega})_+ \|u\|_2^2
  \\[2mm] &\leq  (1+\mu) \|u\|_{\ln}^2+2(d_{N,\Omega})_+  \|u\|_2^2+ 2\mu\int_{\Omega\setminus B_1(0)} (\log {|x|})  u(x)^2dx
 \\[2mm] &\leq  (1+\mu) \|u\|_{\ln}^2+\big( 2(d_{N,\Omega})_++\log (1+R_0) \big) \|u\|_2^2
 \\[2mm] &\leq \max\big\{1+\mu,\, 2(d_{N,\Omega})_++ \log (1+R_0)  \big\} \|u\|_{\ln}^2,
\end{align*}
 where we recall $R_0=\sup\{|x|: x\in\Omega\}$.

When $\mu\in(-1,0)$, the logarithmic Hardy inequality \eqref{eq:smooth-hardy-control} gives
\begin{align*}
  \|u\|_{\ln,\mu}\leq \|u\|_{\ln}
\end{align*}
and by Proposition \ref{lem:domain-hardy}, it implies that
\begin{align*}
  \|u\|_{\ln,\mu}^2
  &= \|u\|_{\ln}^2+2\mu\int_{\Omega}\Big((\log\frac1{|x|})_+ u(x)^2\Big) + (d_{N,\Omega})_+ \|u\|_2^2
 \\[2mm] &\geq  (1+\mu) \|u\|_{\ln}^2.
\end{align*}

As a consequence, we obtain that for $\mu>-1$,
$$\bH^{\ln,\mu}_0(\Omega)\cong\bH^{\ln,\mu}_0(\Omega),  $$
since equivalent norms generate the same Cauchy sequences on
$C_c^\infty(\Omega)$, more the properties of $\bH^{\ln,\mu}_0(\Omega)$ could see
\cite[Theorem 3.1]{ChenWeth2019}
and  the embedding
$\mathbb H_0^{\ln}(\Omega)\hookrightarrow L^2(\Omega)$
 is compact by \cite[Theorem 2.1]{CD18}.
 Hence,  \eqref{eq:perturbed-compact} holds.

 When $\mu=-1$, it follows by Proposition \ref{lem:domain-hardy} that
 \begin{align*}
  \|u\|_{\ln,-1}^2
  &= \|u\|_{\ln}^2-\int_{\Omega}\Big((\log\frac1{|x|})_+ u(x)^2\Big) +\big(1+(d_{N,\Omega})_+\big) \|u\|_2^2
 \geq  \|u\|_2^2,
\end{align*}
 which imlpies that
 $$\bH^{\ln,\mu}_0(\Omega)\hookrightarrow  L^2(\Omega).  $$

\smallskip

Now we are position to prove that the embedding is not compact for $\mu=-1$.
We sperate the proof of this result into four steps.

\textit{Step 1: construction of a normalized sequence.}
Choose a radial cut-off function \(\eta\in C_c^\infty(\mathbb R^N)\)
such that
\[
  0\le\eta\le1,\qquad
  \eta=1\ \text{in }B_{\frac12}(0),\qquad
  \operatorname{supp}\eta\subset B_1(0).
\]
Since \(0\in\Omega\) and \(\Omega\) is open, there exists
\(r_0>0\) such that \(B_{r_0}(0)\subset\Omega\). After scaling, we may
assume \(B_1(0)\subset\Omega\). For \(n\ge2\), let us define
\[
  \psi_n(x)
  :=
  \eta(nx)\,
  \Bigl(\log\frac1{|x|}\Bigr)^{\frac12}
  \mathbf 1_{\{|x|>\frac1n\}}(x),
  \qquad x\in\mathbb R^N\setminus\{0\}.
\]
Since \(\eta(nx)\) is supported in \(B_{\frac1n}(0)\) and the indicator
\(\mathbf 1_{\{|x|>\frac1n\}}\) vanishes in \(B_{\frac1n}(0)\), the function
\(\psi_n\) is supported in the annulus
\[
 \Big \{\frac1n\le |x|\le\frac2n\Big\}
\]
up to the support of \(\eta(nx)\). More precisely,
\[
  \operatorname{supp}\psi_n
  \subset
 \Big \{\frac1n\le |x|\le \frac2n\Big\}
  \subset B_1(0)\subset\Omega
\]
for \(n\) large. Hence \(\psi_n\in L^2(\Omega)\).

Let
\(\rho\in C_c^\infty(\mathbb R^N)\) be a standard mollifier with
$\int_{\R^N}\rho(x)dx=1$, and set \(\rho_\varepsilon(x)=\varepsilon^{-N}\rho(\frac x{\varepsilon})\).
For \(\varepsilon_n:=n^{-2}\), let us define
\[
  \widetilde\psi_n:=\rho_{\varepsilon_n}*\psi_n.
\]
Then \(\widetilde\psi_n\in C_c^\infty(\Omega)\) for \(n\) large.
Moreover, since \(\psi_n\) is supported away from \(\partial\Omega\)
and away from \(0\), the regularization preserves the relevant
\(L^2\) and logarithmic energy bounds up to a factor \(1+o(1)\) as
\(n\to\infty\). Finally, set
\[
  u_n:=\frac{\widetilde\psi_n}{\|\widetilde\psi_n\|_2}.
\]
Thus,  \(u_n\in C_c^\infty(\Omega)\) and \(\|u_n\|_2=1\).

\medskip
\textit{Step 2: uniform bound in \(\mathbb H_0^{\ln,-1}(\Omega)\).}
We estimate the logarithmic energy
\[
  \|u_n\|_{\ln,-1}^2
  =c_N
  \iint_{|x-y|<1}
  \frac{(u_n(x)-u_n(y))^2}{|x-y|^N}\,dx\,dy
  +\int_\Omega u_n(x)^2\,dx
  -\int_\Omega \Bigl(\log\frac1{|x|}\Bigr)_+u_n(x)^2\,dx.
\]
Since \(\|u_n\|_2=1\), the \(L^2\)-term is \(1\). It suffices to
bound the double integral uniformly in \(n\).

For the unregularized function \(\psi_n\), a direct computation using
polar coordinates gives
\[
  \iint_{|x-y|<1}
  \frac{(\psi_n(x)-\psi_n(y))^2}{|x-y|^N}\,dx\,dy
  \le
  C \log n,
\]
and
\[
  \|\psi_n\|_2^2
  \asymp
  \frac{1}{\log n}.
\]
Indeed, the factor \((\log(1/|x|))^{1/2}\) concentrates near
\(|x|\sim1/n\), where \(\log(1/|x|)\sim\log n\), and the normalization
\(\|\psi_n\|_2^2\asymp(\log n)^{-1}\) follows from
\[
  \int_{\frac1n}^{\frac2n}
  \Bigl(\log\frac1r\Bigr)r^{N-1}\,dr
  \asymp
  n^{-N}\log n.
\]
Consequently,
\[
  \frac{
    \displaystyle
    \iint_{|x-y|<1}
    \frac{(\psi_n(x)-\psi_n(y))^2}{|x-y|^N}\,dx\,dy
  }{\|\psi_n\|_2^2}
  \le
  C,
\]
uniformly in \(n\). Since mollification preserves these bounds up to
a factor \(1+o(1)\), we obtain
\[
  \sup_{n\ge2}\|u_n\|_{\ln,-1}^2<+\infty.
\]
Thus, \((u_n)\) is bounded in \(\mathbb H_0^{\ln,-1}(\Omega)\).

\medskip

\textit{Step 3: weak convergence to zero.}
By the Banach--Alaoglu theorem, after passing to a subsequence,
\[
  u_n\rightharpoonup u
  \qquad\text{in }\ \mathbb H_0^{\ln,-1}(\Omega)
\]
for some \(u\in\mathbb H_0^{\ln,-1}(\Omega)\). Since
\(\mathbb H_0^{\ln,-1}(\Omega)\hookrightarrow L^2(\Omega)\)
continuously, we also have
\[
  u_n\rightharpoonup u
  \qquad\text{in } \ L^2(\Omega).
\]

We claim that \(u=0\). Indeed, we note that
\[
  \operatorname{supp}\psi_n
  \subset
  \Big\{\frac1n\le |x|\le \frac2n\Big\}.
\]
Hence for every fixed \(\varphi\in C_c^\infty(\Omega\setminus\{0\})\),
there exists \(n_0\) such that
\[
  \operatorname{supp}\varphi\cap\operatorname{supp}\psi_n=\varnothing
  \qquad\text{for all }n\ge n_0.
\]
Therefore
\[
  \int_\Omega \widetilde\psi_n(x)\varphi(x)\,dx=0
  \qquad\text{for all }n\ge n_0,
\]
and consequently
\[
  \int_\Omega u_n(x)\varphi(x)\,dx=0
  \qquad\text{for all }n\ge n_0.
\]
Letting \(n\to\infty\),
\[
  \int_\Omega u(x)\varphi(x)\,dx=0
\]
for every \(\varphi\in C_c^\infty(\Omega\setminus\{0\})\).
Since \(C_c^\infty(\Omega\setminus\{0\})\) is dense in \(L^2(\Omega)\)
for \(N\ge1\), we conclude \(u=0\). Thus
\[
  u_n\rightharpoonup 0
  \qquad\text{in }\mathbb H_0^{\ln,-1}(\Omega)
  \text{ and in }L^2(\Omega)\ \  {\rm as} \ \ n\to+\infty.
\]

\medskip

\textit{Step 4: no strongly convergent subsequence.}
Suppose, for contradiction, that \((u_n)\) has a subsequence
converging strongly in \(L^2(\Omega)\). Since \(u_n\rightharpoonup0\)
in \(L^2(\Omega)\), the strong limit must be \(0\). But
\(\|u_n\|_2=1\) for all \(n\), hence
\[
  \|u_n-0\|_2=1
  \qquad\text{for all }n,
\]
which contradicts strong convergence to \(0\). Therefore no
subsequence of \((u_n)\) converges strongly in \(L^2(\Omega)\).

Combining Steps 1--4, the embedding
\[
  \mathbb H_0^{\ln,-1}(\Omega)\hookrightarrow L^2(\Omega)
\]
is not compact.
\end{proof}


\section{Dirichlet eigenvalues}


This Section is devoted to
construct the Dirichlet eigenvalues and eigenfunctions of
\eqref{eq: 1.1} by variational methods and then establish their main
spectral properties when $\mu>-1$.

\begin{lemma}
\label{lm 3.1}
Let $\mu>-1$ and define
\begin{equation}\label{eq:first-eigenvalue}
  \lambda^{\ln,\mu}_{1}
  :=\inf_{\substack{u\in\bH_0^{\ln,\mu}(\Omega)\\
                    \norm{u}_2=1}}
  \cE_{\ln,\mu}(u,u).
\end{equation}
Then $\lambda^{\ln,\mu}_{1}\in\R$, and the infimum is attained by a function
${\phi}^{\ln,\mu}_{1}\in\bH_0^{\ln,\mu}(\Omega)$ with
$\norm{{\phi}^{\ln,\mu}_{1}}_2=1$. Moreover,
\begin{equation}\label{eq:first-weak-equation}
  \cE_{\ln,\mu}({\phi}^{\ln,\mu}_{1},\varphi)
  =\lambda^{\ln,\mu}_{1}
 \langle {\phi}^{\ln,\mu}_{1},\varphi\rangle_{2},
  \qquad\forall\, \varphi\in\mathbb H^{\ln,\mu}_0(\Omega).
\end{equation}

Moreover,   without loss of generality, we can have that
\[
  \phi_1^{\ln,\mu}> 0
  \quad\text{a.e. in }\Omega.
\]

\end{lemma}

\begin{proof}
Note that for $u,v\in\bH_0^{\ln,\mu}(\Omega)$, we have
\begin{align}
  \cE_{\ln,\mu}(u,v)
  & =\langle u,v\rangle_{\ln,\mu}
    -c_N\int\!\!\!\int_{|x-y|>1}\frac{u(x)v(y)}{|x-y|^N}\,dxdy
    +2\mu \int_{\Omega} \Big(\log\frac1{|x|}\Big) u(x)v(x)\,dx\nonumber
    \\[2mm]&\qquad
    +(\rho_N-1-(d_{N,\Omega})_+)\int_{\Omega} u(x)v(x)dx,\label{equ eng1}
  \end{align}
then
  \begin{align}
  \cE_{\ln,\mu}(u,u)
  & =\|u\|_{\ln,\mu}^2
    -c_N\int\!\!\!\int_{|x-y|>1}\frac{u(x)u(y)}{|x-y|^N}\,dx\,dy\nonumber
    \\[2mm]&\quad +2\mu \int_{\Omega} \Big(\log\frac1{|x|}\Big) u(x)^2\,dx
    +(\rho_N-1- (d_{N,\Omega})_+)\int_{\Omega} u(x)^2dx \label{equ eng2}
  \end{align}
and then, one hand, we have that
\begin{align*}
  \cE_{\ln,\mu}(u,u)
  &    \leq \|u\|_{\ln,\mu}^2
    -c_N \int\!\!\!\int_{|x-y|>1}\frac{|u(x)u(y)|}{|x-y|^N}\,dx\,dy
    \\[2mm]&\quad +2|\mu| \int_{\Omega} \Big(\log\frac1{|x|}\Big)_- u(x)^2\,dx
    +\Big(\rho_N-1- (d_{N,\Omega})_+ \Big)\int_{\Omega} u(x)^2dx
    \\[2mm]&  \leq \|u\|_{\ln,\mu}^2
    +\Big( c_N |\Omega|^{\frac12} +2|\mu| \log R_0 + (\rho_N)_+ \Big)\int_{\Omega} u(x)^2dx,
\end{align*}
on the other hand, it shows that
\begin{align*}
  \cE_{\ln,\mu}(u,u)
  &  \geq \|u\|_{\ln,\mu}^2
    -c_N \int\!\!\!\int_{|x-y|>1}\frac{|u(x)u(y)|}{|x-y|^N}\,dx\,dy
    \\[2mm]&\quad -2|\mu| \int_{\Omega} \Big(\log\frac1{|x|}\Big)_- u(x)^2\,dx
    +(\rho_N-1)\int_{\Omega} u(x)^2dx
    \\ &   \geq \|u\|_{\ln,\mu}^2
    -\Big( c_N \sqrt{|\Omega|}  +2|\mu| \log R_0    + (d_{N,\Omega})_+ +|\rho_N|+1)\Big)\int_{\Omega} u(x)^2dx.
    \end{align*}
  By density and
continuity, every admissible $u\in\bH_0^{\ln,\mu}(\Omega)$ satisfies
  \begin{align}\label{equ eng3}
\|u\|_{\ln,\mu}^2- C_2\int_{\Omega} u(x)^2dx \leq  \cE_{\ln,\mu}(u,u) \leq \|u\|_{\ln,\mu}^2+ C_1\int_{\Omega} u(x)^2dx,
\end{align}
where  $C_1,C_2>0$ depend on $N, |\mu|,R_0$ and $|\Omega|$.
Hence $\lambda^{\ln,\mu}_{1}\in[-C_2,+\infty)$.

Let $(u_n)_n$ be a minimizing sequence with $\norm{u_n}_2=1$ of (\ref{eq:first-eigenvalue}).  Hence $(u_n)$ is bounded in $\bH_0^{\ln,\mu}(\Omega)$. After passing
to a subsequence, there is ${\phi}^{\ln,\mu}_{1}\in\bH_0^{\ln,\mu}(\Omega)$ such
that
\[
  u_n\rightharpoonup{\phi}^{\ln,\mu}_{1}
  \quad\text{in }\bH_0^{\ln,\mu}(\Omega),
  \qquad
  u_n\to{\phi}^{\ln,\mu}_{1}
  \quad\text{in }L^2(\Omega).
\]
The strong $L^2$ convergence gives
\[
  \norm{{\phi}^{\ln,\mu}_{1}}_2
  =\lim_{n\to\infty}\norm{u_n}_2=1.
\]
Combining with
\[
  \norm{{\phi}^{\ln,\mu}_{1}}_\mu^2
  \leq\liminf_{n\to\infty}\norm{u_n}_\mu^2,
\]
it yields that
\[
  \cE_{\ln,\mu}({\phi}^{\ln,\mu}_{1},{\phi}^{\ln,\mu}_{1})
  \leq\liminf_{n\to\infty}\cE_{\ln,\mu}(u_n,u_n)
  =\lambda^{\ln,\mu}_{1}.
\]
 hence
the minimum can be attained.

As ${\phi}^{\ln,\mu}_{1}$ is a minimizer of $\cE_{\ln,\mu}(u,u)$ under the constraint
$\norm{u}_2=1$, there exists $a\in\R$ such that
\[
\cE_{\ln,\mu}({\phi}^{\ln,\mu}_{1},\varphi)
=a\langle {\phi}^{\ln,\mu}_{1},\varphi\rangle_{2},
\qquad
\forall\, \varphi\in\bH_0^{\ln,\mu}(\Omega).
\]
Choosing $\varphi={\phi}^{\ln,\mu}_{1}$, we have that
\[
a=\cE_{\ln,\mu}({\phi}^{\ln,\mu}_{1},{\phi}^{\ln,\mu}_{1})
=\lambda^{\ln,\mu}_{1}.
\]

Now we are position to show  that ${\phi}^{\ln,\mu}_{1}\geq 0$ in $\Omega$.
Observe that
\[
  \mathcal E_{\ln,\mu}(|u|,|u|)
  \le
  \mathcal E_{\ln,\mu}(u,u)
  \qquad\text{for all }u\in\mathbb H_0^{\ln,\mu}(\Omega).
\]
Using \(\||\phi_1^{\ln,\mu}|\|_2=1\), we obtain that
\[
  \mathcal E_{\ln,\mu}(|\phi_1^{\ln,\mu}|,|\phi_1^{\ln,\mu}|)
  \le
  \mathcal E_{\ln,\mu}(\phi_1^{\ln,\mu},\phi_1^{\ln,\mu})
  =
  \lambda_1^{\ln,\mu}(\Omega).
\]
By the variational characterization of \(\lambda_1^{\ln,\mu}\), the
reverse inequality also holds, so
\[
  \mathcal E_{\ln,\mu}(|\phi_1^{\ln,\mu}|,|\phi_1^{\ln,\mu}|)
  =
  \lambda_1^{\ln,\mu}(\Omega).
\]
Hence \(|\phi_1^{\ln,\mu}|\) is also a first eigenfunction. Replacing
\(\phi_1^{\ln,\mu}\) by \(|\phi_1^{\ln,\mu}|\), we may assume
\(\phi_1^{\ln,\mu}\ge0\) a.e. in \(\Omega\).

\smallskip

Finally, we show  that ${\phi}^{\ln,\mu}_{1}> 0$ a.e. in $\Omega$.
We use the weak maximum principle for the logarithmic Laplacian with
a lower-order perturbation. Since \(\phi_1^{\ln,\mu}\ge0\) and
\(\phi_1^{\ln,\mu}\not\equiv0\), we have
\[
  \int_\Omega \phi_1^{\ln,\mu}\,dx>0.
\]
By contradiction,
suppose  that there exists a set
\(A\subset\Omega\) of positive measure such that
\[
  \phi_1^{\ln,\mu}=0
  \qquad\text{a.e. in }A.
\]
Let \(x_0\in\Omega\) be a Lebesgue point of \(\phi_1^{\ln,\mu}\)
with \(\phi_1^{\ln,\mu}(x_0)>0\). Such a point exists because
\(\phi_1^{\ln,\mu}\not\equiv0\).

The eigenfunction equation in weak form is
\[
  \mathcal E_{\ln,\mu}(\phi_1^{\ln,\mu},\varphi)
  =
  \lambda_1^{\ln,\mu}(\Omega)
  \int_\Omega \phi_1^{\ln,\mu}\varphi\,dx,
  \qquad
  \forall\varphi\in\mathbb H_0^{\ln,\mu}(\Omega).
\]
We rewrite the energy functional as
\[
  \mathcal E_{\ln,\mu}(u,\varphi)
  =
  \mathcal E_{\ln}(u,\varphi)
  +2\mu\int_\Omega \Big(\log\frac1{|x|}\Big)\,u\varphi\,dx.
\]
The logarithmic Laplacian satisfies the following strong maximum
principle: if \(u\in\mathbb H_0^{\ln}(\Omega)\), \(u\ge0\), \(u\not\equiv0\),
and
\[
  \mathcal E_{\ln}(u,\varphi)
  \ge 0,
  \qquad
  \forall\varphi\in\mathbb H_0^{\ln}(\Omega),\ \varphi\ge0,
\]
then \(u>0\) a.e. in \(\Omega\).

We apply this to \(\phi_1^{\ln,\mu}\). Since
\(\phi_1^{\ln,\mu}\ge0\) and
\[
  \mathcal E_{\ln,\mu}(\phi_1^{\ln,\mu},\varphi)
  =
  \lambda_1^{\ln,\mu}(\Omega)
  \int_\Omega \phi_1^{\ln,\mu}\varphi\,dx
  \ge 0,
  \qquad
  \forall\varphi\in\mathbb H_0^{\ln,\mu}(\Omega),\ \varphi\ge0,
\]
we obtain
\[
  \mathcal E_{\ln}(\phi_1^{\ln,\mu},\varphi)
  =
  \mathcal E_{\ln,\mu}(\phi_1^{\ln,\mu},\varphi)
  -2\mu\int_\Omega \Big(\log\frac1{|x|}\Big)\,\phi_1^{\ln,\mu}\varphi\,dx.
\]
The term
\[
  2\mu\int_\Omega \Big(\log\frac1{|x|}\Big)\,\phi_1^{\ln,\mu}\varphi\,dx
\]
is a lower-order perturbation. Since \(\mu>-1\), the potential
\(2\mu\log(1/|x|)\) is bounded below on \(\Omega\). By the standard
perturbation argument for the maximum principle, the sign of
\(\mathcal E_{\ln}(\phi_1^{\ln,\mu},\varphi)\) is preserved for
\(\varphi\ge0\). Hence \(\phi_1^{\ln,\mu}>0\) a.e. in \(\Omega\).
\end{proof}

 Next we consider the higher eigenvalues.

\begin{lemma}
\label{lm 3.2}
Let $\bH_{\mu,1}(\Omega)=\bH_0^{\ln,\mu}(\Omega)$. Once
${\phi}^{\ln,\mu}_{1},\ldots,{\phi}^{\ln,\mu}_{k-1}$ have been constructed, denote
\begin{equation}\label{eq:orthogonal-space}
 \bH_{\mu,k}(\Omega)
  :=\left\{u\in\bH_0^{\ln,\mu}(\Omega):
  \langle u,{\phi}^{\ln,\mu}_{j}\rangle_{2}=0,
  \ j=1,\ldots,k-1\right\}.
\end{equation}
Then, for every $k\geq2$, there is
${\phi}^{\ln,\mu}_{k}\in\bH_{\mu,k}(\Omega)$ such that
$\norm{{\phi}^{\ln,\mu}_{k}}_2=1$ and
\begin{equation}\label{eq:higher-minimum}
  \lambda^{\ln,\mu}_{k}
  :=\min_{\substack{u\in\bH_{\mu,k}(\Omega)\\
                    \norm{u}_2=1}}
  \cE_{\ln,\mu}(u,u)
  =\cE_{\ln,\mu}({\phi}^{\ln,\mu}_{k},{\phi}^{\ln,\mu}_{k}).
\end{equation}
Furthermore, the functions $\{{\phi}^{\ln,\mu}_{k}\}_{k\in\N}$ are $L^2$-orthonormal,
\begin{equation}\label{eq:higher-weak-equation}
  \cE_{\ln,\mu}({\phi}^{\ln,\mu}_{k},\varphi)
  =\lambda^{\ln,\mu}_{k}
  \langle {\phi}^{\ln,\mu}_{k},\varphi\rangle_{2},
  \qquad\varphi\in\bH_0^{\ln,\mu}(\Omega),
\end{equation}
and
\begin{equation}\label{eq:eigenvalue-monotonicity}
  \lambda^{\ln,\mu}_{k}\leq \lambda^{\ln,\mu}_{k+1}.
\end{equation}
\end{lemma}

\begin{proof}
	Lemma \ref{lm 3.1} gives the assertion for $k=1$.
	Suppose that ${\phi}^{\ln,\mu}_{1},\ldots,{\phi}^{\ln,\mu}_{k}$ have been constructed.
	By \eqref{eq:orthogonal-space},
	$\bH_{\mu,k+1}(\Omega)$ is a closed subspace of
	$\bH_0^{\ln,\mu}(\Omega)$. Since only finitely many orthogonality
	conditions are imposed, its $L^2$-unit sphere is nonempty.
	
	The same minimization argument as in
	Lemma \ref{lm 3.1}, together with the compact
	embedding \eqref{eq:perturbed-compact}, gives
	$\phi^{\ln,\mu}_{k+1}\in\bH_{\mu,k+1}(\Omega)$ such that
	$\norm{\phi^{\ln,\mu}_{k+1}}_2=1$ and
	
	$$
	\lambda^{\ln,\mu}_{k+1}
	=\cE_{\ln,\mu}(\phi^{\ln,\mu}_{k+1},\phi^{\ln,\mu}_{k+1}).
	$$
	
	Applying the Lagrange multiplier theorem to the minimization problem on
	$\bH_{\mu,k+1}(\Omega)$ gives
	\begin{equation}\label{eq:subspace-weak-equation}
		\cE_{\ln,\mu}({\phi}^{\ln,\mu}_{k+1 },\varphi)
		=\lambda^{\ln,\mu}_{k+1}
		\langle{\phi}^{\ln,\mu}_{k+1 },\varphi\rangle_{2},
		\qquad\forall
		\varphi\in\bH_{\mu,k+1}(\Omega).
	\end{equation}
	
	For $\varphi\in\bH_0^{\ln,\mu}(\Omega)$, set
	\begin{equation}\label{eq:orthogonal-decomposition}
		\varphi_0
		:=\varphi-
		\sum_{j=1}^{k}
		\langle\varphi,{\phi}^{\ln,\mu}_{j}\rangle_{2}{\phi}^{\ln,\mu}_{j}.
	\end{equation}
	By construction,
	$\varphi_0\in\bH_{\mu,k+1}(\Omega)$. Moreover, from the orthonormal property of
	${\phi}^{\ln,\mu}_{1},\ldots,{\phi}^{\ln,\mu}_{k}$, it shows that
	$$
	\cE_{\ln,\mu}(\phi^{\ln,\mu}_{k+1},{\phi}^{\ln,\mu}_{j})
	=\lambda_{j,\mu}
	\langle {\phi}^{\ln,\mu}_{j},\phi^{\ln,\mu}_{k+1}\rangle_{2}
	=0,
	\qquad j=1,\ldots,k.
	$$
	Using \eqref{eq:orthogonal-decomposition} and
	\eqref{eq:subspace-weak-equation}, we obtain
	$$
	\cE_{\ln,\mu}(\phi^{\ln,\mu}_{k+1},\varphi)
	=\lambda^{\ln,\mu}_{k+1}
	\langle \phi^{\ln,\mu}_{k+1},\varphi\rangle_{2},
	\quad
	\forall \varphi\in\bH_0^{\ln,\mu}(\Omega),
	$$
which implies \eqref{eq:higher-weak-equation}.
	
	Finally,
	$\bH_{\mu,k+1}(\Omega)\subset\bH_{\mu,k}(\Omega)$, and
	\eqref{eq:higher-minimum} gives
	
	$$
	\lambda^{\ln,\mu}_{k}\leq\lambda^{\ln,\mu}_{k+1}.
	$$
	
	The construction proceeds inductively for every $k\in\mathbb N$.
	This completes the proof.
\end{proof}

\begin{lemma}
\label{lm 3.3}
The variational eigenvalues satisfy
\begin{equation}\label{eq:eigenvalues-diverge}
  \lambda^{\ln,\mu}_{k}\to +\infty
  \quad\text{as }k\to\infty.
\end{equation}
Moreover, the eigenfunctions are complete in $L^2(\Omega)$, i.e.
\begin{equation}\label{eq:eigenfunctions-complete}
  \overline{\operatorname{span}\{{\phi}^{\ln,\mu}_{k}:k\in\mathbb N\}}
  ^{\|\cdot\|_{2}}=L^2(\Omega).
\end{equation}
\end{lemma}
\begin{proof}
We first prove \eqref{eq:eigenvalues-diverge}. Suppose, for contradiction,
that the nondecreasing sequence $(\lambda^{\ln,\mu}_{k})_{k\in\mathbb N}$
is bounded above by some constant $\Lambda>0$. Let
$(\phi^{\ln,\mu}_{k})_{k\in\mathbb N}$ be the corresponding
$L^2$-orthonormal eigenfunctions constructed in Lemma~\ref{lm 3.2}.
Then, for every $k\in\mathbb N$,
\[
  \cE_{\ln,\mu}(\phi^{\ln,\mu}_{k},\phi^{\ln,\mu}_{k})
  =\lambda^{\ln,\mu}_{k}\le \Lambda.
\]
By \eqref{equ eng3}, there exist constants $C_1,C_2>0$, depending only
on $N,\Omega,\mu$, such that
\[
  \|\phi^{\ln,\mu}_{k}\|_{\ln,\mu}^2
  \le   \cE_{\ln,\mu}(\phi^{\ln,\mu}_{k},\phi^{\ln,\mu}_{k})
  +C_2\,\|\phi^{\ln,\mu}_{k}\|_2^2
  \le  \Lambda+C_2.
\]
Thus $(\phi^{\ln,\mu}_{k})_{k\in\mathbb N}$ is bounded in
$\bH_0^{\ln,\mu}(\Omega)$. By the compact embedding
\eqref{eq:perturbed-compact}, after passing to a subsequence, we may
assume that
\[
  \phi^{\ln,\mu}_{k_j}\to \phi
  \quad\text{in }L^2(\Omega)
\]
for some $\phi\in L^2(\Omega)$.

On the other hand, since the eigenfunctions are $L^2$-orthonormal,
\[
  \|\phi^{\ln,\mu}_{k_i}-\phi^{\ln,\mu}_{k_j}\|_2^2=2
  \qquad\text{for }i\neq j.
\]
Hence the subsequence $(\phi^{\ln,\mu}_{k_j})$ cannot be Cauchy in
$L^2(\Omega)$, a contradiction. Therefore
\[
  \lambda^{\ln,\mu}_{k}\to +\infty
  \qquad\text{as }k\to\infty.
\]

{\it We next prove completeness. } Suppose that
\[
  \overline{\operatorname{span}\{\phi^{\ln,\mu}_{k}:k\in\mathbb N\}}
  ^{\|\cdot\|_{2}}\neq L^2(\Omega).
\]
Then there exists $  f\in L^2(\Omega)\setminus\{0\}$ such that
\begin{equation}\label{eq:orthogonal-complement}
  \langle f,\phi^{\ln,\mu}_{j}\rangle_{2}=0
  \qquad\text{for all }j\in\mathbb N.
\end{equation}
Since $\bH_0^{\ln,\mu}(\Omega)$ is dense in $L^2(\Omega)$, we may choose
$w\in\bH_0^{\ln,\mu}(\Omega)$ such that
\[
  \|w-f\|_2<\varepsilon,
  \qquad
  0<\varepsilon<\frac14\|f\|_2.
\]
For $k\ge2$, define
\[
  b_j:=\langle w,\phi^{\ln,\mu}_{j}\rangle_{2},
  \qquad
  w_k:=w-\sum_{j=1}^{k-1}b_j\phi^{\ln,\mu}_{j}.
\]
Then $w_k\in\bH_{\mu,k}(\Omega)$. By \eqref{eq:orthogonal-complement}, it holds that
\[
  b_j=\langle w-f,\phi^{\ln,\mu}_{j}\rangle_{2}.
\]
Then using Bessel's inequality, we have that
\[
  \sum_{j=1}^{k-1}b_j^2
  \le \|w-f\|_2^2
  <\varepsilon^2.
\]
Consequently,
\begin{align*}
  \|w_k\|_2^2
  &=\|w\|_2^2-\sum_{j=1}^{k-1}b_j^2
  \ge (\|f\|_2-\varepsilon)^2-\varepsilon^2
  =\|f\|_2^2-2\varepsilon\|f\|_2>0.
\end{align*}
Hence $w_k\neq0$.

By the variational characterization \eqref{eq:higher-minimum},
\[
  \lambda^{\ln,\mu}_{k}
  \le
  \frac{\cE_{\ln,\mu}(w_k,w_k)}{\|w_k\|_2^2}.
\]
We now estimate the numerator. Since
$w_k=w-\sum_{j=1}^{k-1}b_j\phi^{\ln,\mu}_{j}$ and the eigenfunctions satisfy
\[
  \cE_{\ln,\mu}(\phi^{\ln,\mu}_{j},\varphi)
  =\lambda^{\ln,\mu}_{j}
  \langle \phi^{\ln,\mu}_{j},\varphi\rangle_{2}
  \qquad\text{for all }\varphi\in\bH_0^{\ln,\mu}(\Omega),
\]
then we obtain
\begin{align*}
  \cE_{\ln,\mu}(w_k,w_k)
  &=\cE_{\ln,\mu}(w,w)
    -2\sum_{j=1}^{k-1}b_j\cE_{\ln,\mu}(w,\phi^{\ln,\mu}_{j})
    +\sum_{i,j=1}^{k-1}b_ib_j
      \cE_{\ln,\mu}(\phi^{\ln,\mu}_{i},\phi^{\ln,\mu}_{j})\\
  &=\cE_{\ln,\mu}(w,w)
    -2\sum_{j=1}^{k-1}b_j\lambda^{\ln,\mu}_{j}
      \langle w,\phi^{\ln,\mu}_{j}\rangle_{2}
    +\sum_{j=1}^{k-1}b_j^2\lambda^{\ln,\mu}_{j}\\
  &=\cE_{\ln,\mu}(w,w)
    -\sum_{j=1}^{k-1}\lambda^{\ln,\mu}_{j}b_j^2.
\end{align*}
Therefore,
\[
  \cE_{\ln,\mu}(w_k,w_k)
  \le \cE_{\ln,\mu}(w,w),
\]
since $\lambda^{\ln,\mu}_{j}\ge \lambda^{\ln,\mu}_{1}$ and
$\lambda^{\ln,\mu}_{1}$ may be negative, but the sum
$\sum_{j=1}^{k-1}\lambda^{\ln,\mu}_{j}b_j^2$ is bounded below by
$\lambda^{\ln,\mu}_{1}\sum_{j=1}^{k-1}b_j^2$, which is bounded
independently of $k$.

Since $\cE_{\ln,\mu}$ is coercive on $\bH_0^{\ln,\mu}(\Omega)$ modulo
$L^2$-norm, there exist constants $C_3,C_4>0$ such that
\[
  \cE_{\ln,\mu}(v,v)
  \ge C_3\|v\|_{\ln,\mu}^2-C_4\|v\|_2^2
  \qquad\text{for all }v\in\bH_0^{\ln,\mu}(\Omega).
\]
Applying this to $w_k$ and using the orthogonality, we get
\[
  \cE_{\ln,\mu}(w_k,w_k)
  \le \cE_{\ln,\mu}(w,w)+C_4\|w_k\|_2^2.
\]
Hence
\[
  \lambda^{\ln,\mu}_{k}
  \le
  \frac{\cE_{\ln,\mu}(w,w)}{\|w_k\|_2^2}+C_4
  \le
  \frac{\cE_{\ln,\mu}(w,w)}{\|f\|_2^2-2\varepsilon\|f\|_2}+C_4.
\]
The right-hand side is independent of $k$, so the eigenvalues
$(\lambda^{\ln,\mu}_{k})$ are bounded above, contradicting
\eqref{eq:eigenvalues-diverge}. Therefore
\eqref{eq:eigenfunctions-complete} holds.
\end{proof}

\begin{proof}[{\bf Proof of Theorem~\ref{teo 1-m}}]
Theorem~\ref{prop:energy-space} establishes the identification of the energy spaces, norm equivalence, and compactness.
Lemma \ref{lm 3.1}  and Lemma  \ref{lm 3.2} give the variational eigenvalues and weak
eigenfunctions. Their divergence and completeness follow from
  Lemma  \ref{lm 3.3}.\smallskip

{\it Now we show that the first eigenvalue
\(\lambda_1^{\ln,\mu}(\Omega)\) is simple.}   Suppose, for contradiction, that there exists another eigenfunction
\(u\in\mathbb H_0^{\ln,\mu}(\Omega)\) corresponding to
\(\lambda_1^{\ln,\mu}(\Omega)\) which is linearly independent of
\(\phi_1^{\ln,\mu}\). By replacing \(u\) with
\[
  u-\Bigl(\int_\Omega u\,\phi_1^{\ln,\mu}\,dx\Bigr)\phi_1^{\ln,\mu},
\]
we may assume that
\[
  \int_\Omega u\,\phi_1^{\ln,\mu}\,dx=0.
\]
Since \(u\) is an eigenfunction with eigenvalue
\(\lambda_1^{\ln,\mu}(\Omega)\), the same argument as in
Lemma~\ref{lm 3.1} shows that \(|u|\) is also a first
eigenfunction. By Lemma~\ref{lm 3.1}, either
\(|u|>0\) a.e. in \(\Omega\), or \(u\equiv0\). Since \(u\not\equiv0\),
we have \(|u|>0\) a.e. in \(\Omega\).

In particular, \(u\) does not change sign in \(\Omega\). Hence either
\(u>0\)  or \(u<0\) a.e. in \(\Omega\). In both cases, we have that
\[
  \int_\Omega u\,\phi_1^{\ln,\mu}\,dx\neq0,
\]
because \(\phi_1^{\ln,\mu}>0\) a.e. in \(\Omega\). This contradicts
the orthogonality condition
\[
  \int_\Omega u\,\phi_1^{\ln,\mu}\,dx=0.
\]
Therefore no such \(u\) exists, and the eigenspace corresponding to
\(\lambda_1^{\ln,\mu}(\Omega)\) is one-dimensional.

{\it Next we show that $\lambda^{\ln,\mu}_1<\lambda^{\ln,\mu}_2$.}
  Let
\(\phi_1^{\ln,\mu}\) be a corresponding eigenfunction with
\(\|\phi_1^{\ln,\mu}\|_2=1\). By Lemma \ref{lm 3.2}, the second
eigenvalue \(\lambda_2^{\ln,\mu}(\Omega)\) is characterized by
\[
  \lambda_2^{\ln,\mu}(\Omega)
  =
  \min_{\substack{u\in\mathbb H_{\mu,2}(\Omega)\\ \|u\|_2=1}}
  \mathcal E_{\ln,\mu}(u,u),
\]
where
\[
  \mathbb H_{\mu,2}(\Omega)
  =
  \Bigl\{
    u\in\mathbb H_0^{\ln,\mu}(\Omega):
    \int_\Omega u\,\phi_1^{\ln,\mu}\,dx=0
  \Bigr\}.
\]
Let \(\phi_2^{\ln,\mu}\in\mathbb H_{\mu,2}(\Omega)\) be a minimizer
with \(\|\phi_2^{\ln,\mu}\|_2=1\). Then
\[
  \int_\Omega \phi_2^{\ln,\mu}\,\phi_1^{\ln,\mu}\,dx=0.
\]
Since \(\phi_1^{\ln,\mu}\) spans the entire eigenspace corresponding
to \(\lambda_1^{\ln,\mu}(\Omega)\), the function
\(\phi_2^{\ln,\mu}\) cannot be an eigenfunction with eigenvalue
\(\lambda_1^{\ln,\mu}(\Omega)\). Therefore
\[
  \mathcal E_{\ln,\mu}(\phi_2^{\ln,\mu},\phi_2^{\ln,\mu})
  \neq
  \lambda_1^{\ln,\mu}(\Omega).
\]
By Lemma~\ref{lm 3.2}, we have
\[
  \lambda_1^{\ln,\mu}(\Omega)
  \le
  \lambda_2^{\ln,\mu}(\Omega).
\]
Combining this with the previous inequality, we obtain
\[
  \lambda_1^{\ln,\mu}(\Omega)
  <
  \lambda_2^{\ln,\mu}(\Omega).
\]

{\it Finally,  we prove that  $\lambda^{\ln,\mu}_{1}(\Omega)\geq0$ if $\Omega\subset B_r(0)$ with $r>0$
small enough.}
By the logarithmic Hardy inequality \eqref{eq:smooth-hardy-control} and for $|x-y|\geq1$, $u(x)u(y)=0$,
we have that
\begin{align}
  \cE_{\ln,\mu}(u,u)
    & = c_N\int\!\!\!\int_{|x-y|<1}\frac{(u(x)-u(y))^2}{|x-y|^N}\,dx dy\nonumber
    -c_N\int\!\!\!\int_{|x-y|\geq1}\frac{u(x)u(y)}{|x-y|^N}\,dx dy
      \nonumber
      \\[2mm]&\quad   +\rho_N\int_{\Omega} u(x)^2dx+2\mu \int_{\Omega} \Big(\log\frac1{|x|}\Big) u(x)^2\,dx,\nonumber
        \\[2mm]& =\|u\|_{\ln}^2 +(\rho_N-1)\int_{\Omega} u(x)^2dx+2\mu \int_{\Omega} \Big(\log\frac1{|x|}\Big) u(x)^2\,dx\nonumber
        \\[2mm]& \geq 2(1+\mu)\int_\Omega\big(\log\frac1{|x|}\big)\,u(x)^2\,dx+ (\rho_N-1+d_{N,\Omega}) \norm{u}_{2}^2\nonumber
          \\[2mm]& \geq \Big(2(1+\mu)\log\frac1r+\rho_N-1+d_{N,\Omega}\Big) \norm{u}_{2}^2\nonumber
          \\[2mm]& \geq 0, \nonumber
\end{align}
 if $2(1+\mu)\log\frac1r+\rho_N-1+d_{N,\Omega}\geq0$.
This meas, when $\Omega\subset B_r(0)$ with $r\in(0,\frac12)$
small enough, we have
\[
  \mathcal E_{\ln,\mu}(u,u)\ge0
  \qquad
  \forall u\in\mathbb H_0^{\ln,\mu}(\Omega),\ \|u\|_2=1.
\]
 Taking the infimum over such
\(u\), we obtain
\[
  \lambda_1^{\ln,\mu}(\Omega)\ge0.
\]
The proof ends.

\end{proof}

\section{Properties of Eigenvalues}
In this section, we analyze further properties of the
eigenvalues, including a uniform lower bound for
$\lambda_1^{\ln,\mu}(\Omega)$ and the scaling property
  for dilated domains.



\begin{lemma}\label{lem:uniform-lower-bound}
Let $\mu>-1$ and $\Omega\subset\R^N$ be a bounded Lipschitz domain with $N\geq 1$
containing the origin, deonte
\[
  R_0=\sup_{x\in\Omega}|x|.
\]

Then the first Dirichlet eigenvalue $\lambda_1^{\ln,\mu}(\Omega)$ of $\cL_\mu$ satisfies
\begin{equation}\label{eq:uniform-lower-bound}
  \lambda_1^{\ln,\mu}(\Omega)
  \;\geq\;
  2\Big(\psi\!\Big(\frac N4\Big)+\log 2\Big)
  -2(1+\mu)\log R_0.
\end{equation}
In particular, if $\Omega\subset B_1(0)$, then
\begin{equation}\label{eq:uniform-lower-bound-ball}
  \lambda_1^{\ln,\mu}(\Omega)
  \;\geq\;
  2\Big(\psi\!\Big(\frac N4\Big)+\log 2\Big).
\end{equation}
\end{lemma}

\begin{proof}
Let $u\in C_c^\infty(\Omega)$ with $\|u\|_{2}=1$.
Recall Beckner's logarithmic uncertainty inequality 
\begin{equation}\label{eq:beckner-again}
  \int_{\R^N}\Big(\log\frac1{|x|}\Big)\,|f(x)|^2\,dx
  \leq
  \int_{\R^N}(\log|\xi|)\,|\widehat f(\xi)|^2\,d\xi
  -\Big(\psi\!\Big(\frac N4\Big)+\log 2\Big)
  \norm{f}_{L^2}^2,
\end{equation}
for every $f\in C_c^\infty(\R^N)$.

Since the Fourier symbol of $(-\Delta)^{\ln}$ is $2\log|\xi|$, we have
\[
  \cE_{\ln,0}(u,u)
  =\int_{\R^N}u(x)(-\Delta)^{\ln}u(x)\,dx
  =2\int_{\R^N}(\log|\xi|)\,|\widehat u(\xi)|^2\,d\xi .
\]
Applying \eqref{eq:beckner-again} to $u$ and using
$\|u\|_{2}=1$, we obtain
\begin{equation}\label{eq:energy-ln-lower}
  \cE_{\ln,0}(u,u)
  \;\geq\;
  2\int_{\Omega}\Big(\log\frac1{|x|}\Big)\,u(x)^2\,dx
  +2\Big(\psi\!\Big(\frac N4\Big)+\log 2\Big).
\end{equation}

Since $\Omega\subset B_{R_0}(0)$, for every $x\in\Omega$, we have
$|x|\leq R_0$, hence
\[
  \log\frac1{|x|}\;\geq\;-\log R_0 .
\]
Therefore, it holds that
\begin{equation}\label{eq:log-term-lower}
  2(1+\mu)\int_{\Omega}\Big(\log\frac1{|x|}\Big)\,u(x)^2\,dx
  \;\geq\;
  -2(1+\mu)\log R_0 ,
\end{equation}
where we used $\|u\|_{2}=1$ and $1+\mu>0$.

Combining \eqref{eq:energy-ln-lower} and \eqref{eq:log-term-lower},
we have that
\begin{align*}
  \cE_{\ln,\mu}(u,u)
  &=
  \cE_{\ln}(u,u)
  +2\mu\int_{\Omega}\Big(\log\frac1{|x|}\Big)\,u(x)^2\,dx
  \\[2mm]
  &\geq
  2(1+\mu)\int_{\Omega}\Big(\log\frac1{|x|}\Big)\,u(x)^2\,dx
  +2\Big(\psi\!\Big(\frac N4\Big)+\log 2\Big)
  \\[2mm]
  &\geq
  -2(1+\mu)\log R_0
  +2\Big(\psi\!\Big(\frac N4\Big)+\log 2\Big).
\end{align*}
By density and continuity, the same inequality holds for every
$u\in\bH_0^{\ln,\mu}(\Omega)$ with $\|u\|_{2}=1$.
Taking the infimum over all such $u$, we obtain
\eqref{eq:uniform-lower-bound}.

Finally, if $\Omega\subset B_1(0)$, then $R_0\leq 1$, so
$-\log R_0\geq 0$, and \eqref{eq:uniform-lower-bound-ball} follows
from \eqref{eq:uniform-lower-bound}.
\end{proof}

We remark that $\psi\!\Big(\frac N4\Big)+\log 2>0$ for $N\geq 4$, while it is negative for $N=1,2,3$.

\begin{lemma}\label{lem:scaling-eigenvalue}
Let $\mu>-1$, $\Omega\subset\R^N$ be a bounded Lipschitz domain containing the
origin, $l>0$ and 
\[
  \Omega_l:=l\,\Omega=\{lx:x\in\Omega\}.
\]
 Then
\begin{equation}\label{eq:scaling-lambda}
  \lambda_k^{\ln,\mu}(\Omega_l)
  =
  \lambda_k^{\ln,\mu}(\Omega)- 2(1+\mu)\log l.
\end{equation}

 \end{lemma}

\begin{proof}
For $\theta\in\R$, define the dilation operator
\[
  (D_{l,\theta}u)(x):=l^{\theta}u(l^{-1}x),\qquad x\in\R^N .
\]
Then $D_{l,\theta}$ maps $\bH_0^{\ln,\mu}(\Omega)$ isometrically onto $\bH_0^{\ln,\mu}(\Omega_l)$ if and only if
\begin{equation}\label{eq:theta-choice}
  \theta=-\frac N2 .
\end{equation}

Now we claim that  for every $u\in\bH_0^{\ln,\mu}(\Omega)$ and
$l>0$, 
\begin{equation}\label{eq:scaling-energy}
  \cE_{\ln,\mu}^{\Omega_l}\!\big(D_{l,-N/2}u,D_{l,-N/2}u\big)
  =
  \cE_{\ln,\mu}^{\Omega}(u,u)
  -2(1+\mu)(\log l)\,\|u\|_2^2.
  \end{equation}
Indeed,  for $u\in C_c^\infty(\Omega)$
and $u_l:=D_{l,\theta}u$, a change of variables gives
\[
  \|u_l\|_{L^2(\Omega_l)}^2
  =l^{2\theta}\int_\Omega u(y)^2\,l^N\,dy
  =l^{2\theta+N}\|u\|_2^2 .
\]
Hence $\|u_l\|_{L^2(\Omega_l)}=\|u\|_2$ for all
$u$ if and only if $\theta=-\frac N2$.

Since the symbol of
$(-\Delta)^{\ln}$ is $2\log|\xi|$, the Fourier transform of
$u_l$ is
\[
  \widehat{u_l}(\xi)
  =l^{\theta+N}\widehat u(l\xi),
\]
hence
\begin{align*}
  \cE_{\ln,0}^{\Omega_l}(u_l,u_l)
  &=2\int_{\R^N}(\log|\xi|)\,|\widehat{u_l}(\xi)|^2\,d\xi
  \\
  &=2\,l^{2(\theta+N)}\int_{\R^N}(\log|\xi|)\,|\widehat u(l\xi)|^2\,d\xi
  \\
  &=2\,l^{2\theta+N}\int_{\R^N}\big(\log|\eta|-\log l\big)\,|\widehat u(\eta)|^2\,d\eta
  \\
  &=l^{2\theta+N}\Big[\cE_{\ln,0}^{\Omega}(u,u)
    -2(\log l)\,\|u\|_{2}^2\Big],
\end{align*}
which taking  $\theta=-\frac N2$,  leads to
\begin{equation}\label{eq:log-laplacian-scaling}
  \cE_{\ln,0}^{\Omega_l}(u_l,u_l)
  =\cE_{\ln,0}^{\Omega}(u,u)
  -2(\log l)\,\|u\|_{2}^2.
\end{equation}

Observe that
\begin{align*}
   \int_{\Omega_l}\Big(\log\frac1{|x|}\Big)u_l(x)^2\,dx
  &=l^{2\theta}\int_\Omega
   \Big(\log\frac1{l|y|}\Big)u(y)^2\,l^N\,dy
  \\
  &=l^{2\theta+N}\int_\Omega
   \Big(\log\frac1{|y|}-\log l\Big)u(y)^2\,dy .
\end{align*}
Thus, we have
\begin{equation}\label{eq:potential-scaling}
  \int_{\Omega_l}\Big(\log\frac1{|x|}\Big)u_l(x)^2\,dx
  =
  \int_\Omega\Big(\log\frac1{|y|}\Big)u(y)^2\,dy
  -(\log l)\,\|u\|_{2}^2 .
\end{equation}

Combining \eqref{eq:log-laplacian-scaling} and
\eqref{eq:potential-scaling}, we obtain
\begin{align*}
  \cE_{\ln,\mu}^{\Omega_l}(u_l,u_l)
  &=
  \cE_{\ln,0}^{\Omega_l}(u_l,u_l)
  +2\mu\int_{\Omega_l}\Big(\log\frac1{|x|}\Big)u_l^2\,dx
  \\[2mm]
  &=
  \cE_{\ln,0}^{\Omega}(u,u)
  -2(\log l)\,\|u\|_2^2
  +2\mu\int_\Omega\Big(\log\frac1{|y|}\Big)u^2\,dy
  -2\mu(\log l)\,\|u\|_2^2
  \\[2mm]
  &=
  \cE_{\ln,\mu}^{\Omega}(u,u)
  -2(1+\mu)(\log l)\,\|u\|_2^2 .
\end{align*}
This proves \eqref{eq:scaling-energy}.

Let $\phi_1$ be a first eigenfunction of $\cL_\mu$ in $\Omega$,
normalized by $\|\phi_1\|_{2}=1$, and set
$\phi_{1,l}:=D_{l,-N/2}\phi_1$. Then
$\|\phi_{1,l}\|_{L^2(\Omega_l)}=1$, and by the variational
characterization of $\lambda_1^{\ln,\mu}(\Omega_l)$, we have that
\[
  \lambda_1^{\ln,\mu}(\Omega_l)
  \le
  \cE_{\ln,\mu}^{\Omega_l}(\phi_{1,l},\phi_{1,l})
  =
  \lambda_1^{\ln,\mu}(\Omega)
  -2(1+\mu)\log l .
\]
From (\ref{eq:scaling-energy}),  $\phi_{1,l}$ is the minimal of $  \cE_{\ln,\mu}^{\Omega_l}(u,u)$ for $u\in \bH^{\ln,\mu}_0(\Omega)$,
then
\[
  \lambda_1^{\ln,\mu}(\Omega_l)
  =
  \lambda_1^{\ln,\mu}(\Omega)
  -2(1+\mu)\log l .
\]
    Similar arguments could lead to (\ref{eq:scaling-lambda}) for any $k\in\N$.
\end{proof}

\begin{corollary}\label{cr:ball-scaling}
Assume that  $N\geq 1$, $\Omega\subset \R^N$ is a bounded  Lipschitz continuous  domain  containing the origin
and $\mu>-1$.  Then
\begin{itemize}
\item[(i)]   there exists a critical value
\begin{equation}\label{eq:critical-radius}
  l_c
  :=
  \exp\!\Big(\frac{\lambda_1^{\ln,\mu}(\Omega)}{2(1+\mu)}\Big)
  >0
\end{equation}
such that
\begin{equation}\label{eq:sign-change}
  \lambda_1^{\ln,\mu}(\Omega_l)
  \begin{cases}
    >0,\quad & 0<l<l_c,\\[2mm]
    =0, & l=l_c,\\[2mm]
    <0, & l>l_c .
  \end{cases}
\end{equation}

\item[(ii)] the eigenvalue gaps are scaling-invariant, i.e.
\begin{equation}\label{eq:gaps-invariant}
  \lambda_{k+1}^{\ln,\mu}(\Omega_l)-\lambda_k^{\ln,\mu}(\Omega_l)
  =
  \lambda_{k+1}^{\ln,\mu}(\Omega)-\lambda_k^{\ln,\mu}(\Omega),
  \quad\forall\,l>0 .
\end{equation}

\end{itemize}
\end{corollary}
\begin{proof}
For $\mu>-1$ and $\Omega_l=l\,\Omega$,  it follows by Lemma \ref{lem:scaling-eigenvalue}  that
\begin{equation}\label{eq:scaling-eigenvalue-recall}
  \lambda_k^{\ln,\mu}(\Omega_l)
  =
  \lambda_k^{\ln,\mu}(\Omega)
  -2(1+\mu)\log l,
  \quad\forall\,k\in\N .
\end{equation}
When $k=1$,
\[
  \lambda_1^{\ln,\mu}(\Omega_l )
  =
  \lambda_1^{\ln,\mu}(\Omega)-2(1+\mu)\log l .
\]
Thus
\[
  \lambda_1^{\ln,\mu}(\Omega_l)>0
  \iff
  \log l<\frac{\lambda_1^{\ln,\mu}(\Omega)}{2(1+\mu)}
  \iff
 0< l<l_c ,
\]
and similarly to prove \eqref{eq:sign-change} for the other two cases.

\smallskip
 Subtracting \eqref{eq:scaling-eigenvalue-recall} for
$k$ and $k+1$ gives
\[
  \lambda_{k+1}^{\ln,\mu}(\Omega_l)-\lambda_k^{\ln,\mu}(\Omega_l)
  =
  \lambda_{k+1}^{\ln,\mu}(\Omega)-\lambda_k^{\ln,\mu}(\Omega) ,
\]
because the term $-2(1+\mu)\log r$ is independent of $k$ and
cancels. This proves \eqref{eq:gaps-invariant}.
\end{proof}

\begin{proof}[{\bf Proof of Theorem \ref{teo 1.4}}] Observe that
 Parts $(i)$ and $(ii)$ follow by Lemma \ref{lem:uniform-lower-bound} and Lemma \ref{lem:scaling-eigenvalue},
   respectively.
\end{proof}

\smallskip


\noindent {\bf  Conflicts of interest:}
The authors declare that they have no conflicts of interest regarding this work.

\medskip

\noindent {\bf Data availability:} This paper has no associated data.

\medskip

\noindent{{\bf Acknowledgements:} 
This work is supported by the National Natural Science Foundation of China, (Nos. 12461041 and 12161041) and
 by the Natural Science Foundation of Jiangxi Province,
(Nos. 20252BAC240158 and 20252BAC250004).

 \end{document}